\documentclass[reqno, 12pt]{amsart}

\usepackage{amsfonts,amsthm,amsmath,amssymb,amscd,mathrsfs}
\usepackage[colorlinks=true,linkcolor=blue,citecolor=red,urlcolor=black]{hyperref}
\usepackage{graphics}
\usepackage{indentfirst}
\usepackage{cite}
\usepackage{latexsym}
\usepackage[dvips]{epsfig}
\usepackage{color}
\usepackage{lmodern}
\usepackage[alphabetic,nobysame]{amsrefs}
\def\MR#1{}
\usepackage[utf8]{inputenc}
\usepackage[T1]{fontenc}
\usepackage{soul}

\newtheorem{theorem}{Theorem}

\newtheorem{corollary}[theorem]{Corollary}

\newcommand{\beq}{\begin{equation}}
	\newcommand{\eeq}{\end{equation}}
\newcommand{\ben}{\begin{eqnarray}}
	\newcommand{\een}{\end{eqnarray}}
\newcommand{\beno}{\begin{eqnarray*}}
	\newcommand{\eeno}{\end{eqnarray*}}
\newcommand{\no}{\nonumber}

\numberwithin{equation}{section}
\newtheorem{Thm}{Theorem}[section]
\newtheorem{Lem}{lemma}[section]
\newtheorem{Rem}{Remark}[section]

\begin{document}
	\title[Saint-Venant Estimates and Liouville-Type Theorems]{\bf Saint-Venant Estimates and Liouville-Type Theorems for the Stationary MHD Equation in \(\mathbb{R}^3\)}
	
	    \author{Jing~Loong}
	    \address[Jing~Loong]{School of Mathematical Sciences, Dalian University of Technology, Dalian, 116024,  China}
	    \email{loongjing@mail.dlut.edu.cn}

		\author{Guoxu~Yang}
		\address[Guoxu~Yang]{School of Mathematical Sciences, Dalian University of Technology, Dalian, 116024,  China}
		\email{guoxu\_dlut@outlook.com}

	\date{\today}

	\begin{abstract} 
		In this paper, we investigate a Liouville-type theorem for the MHD equations using Saint-Venant type estimates. We show that \( (u, B) \) is a trivial solution if the growth of the \( L^s \) mean oscillation of the potential functions for both the velocity and magnetic fields are controlled. Our growth assumption is weaker than those previously known for similar results. The main idea is to refine the Saint-Venant type estimates using the Froullani integral. 
	\end{abstract}
	
	\maketitle
	

	\section{Introduction}
	Consider the following general stationary MHD system in $ \mathbb{R}^3$: 
	\begin{equation}
		\label{eq:MHD}
		\left\{\begin{array}{l}
			-\kappa \Delta u+u \cdot \nabla u+\nabla p=B \cdot \nabla B, \\
			-\nu \Delta B+u \cdot \nabla B-B \cdot \nabla u=0, \\
			\operatorname{div} u=0, \quad \operatorname{div} B=0,
		\end{array}\right.
	\end{equation}
	where $u =(u_1,u_2,u_3)$ is the velocity field of the fluid flows, $B  = (B_1,B_2,B_3)$ is the magnetic field, and $p $ is the pressure of the flows. In addition, $\kappa>0$ and $\nu>0$ are given parameters denoting the fluid viscosity and the magnetic resistivity, respectively. Without loss of generality, let $\kappa=\nu=1$ now on.
	
		In this paper, we focus on the Liouville-type properties of the system \eqref{eq:MHD}, which is
	motivated by the development of Navier-Stokes equations. When $B\equiv0$, \eqref{eq:MHD} is reduced to the Navier-Stokes system in $ \mathbb{R}^3$:
	\begin{equation}
				\label{ns}
				\left\{\begin{array}{l}
						u \cdot \nabla u+\nabla p- \Delta u=0, \\
						\nabla \cdot u=0 ,
					\end{array}\right.
			\end{equation}
	and a very
	challenging open question is whether there exists a nontrivial solution when the Dirichlet integral $\int_{\mathbb{R}^3}|\nabla u|^2dx$ is finite, which
	dates back to Leray's celebrated paper \cite{L1933} and is explicitly written in Galdi's book (\cite{Galdi}, Remark X. 9.4; see also Tsai's book \cite{T2018}).
	The Liouville type problem without any other assumptions remains widely open. Galdi proved the above Liouville type theorem by assuming $u\in L^{\frac92}(\mathbb{R}^3)$ in \cite{Galdi}.
	Chae \cite{Chae2014} showed the condition $\triangle u\in L^{\frac65}(\mathbb{R}^3)$ is sufficient for the vanishing property of $u$ by exploring the maximum principle of the head pressure. Chae-Wolf \cite{ChaeWolf2016} gave an improvement of logarithmic form for Galdi's result
	by assuming that $\int_{\mathbb{R}^3} |u|^{\frac92}\{\ln(2+\frac{1}{|u|})\}^{-1}dx<\infty$.
	Seregin obtained the conditional criterion $u\in BMO^{-1}(\mathbb{R}^3)$ in \cite{Seregin2016}, where $u \in B M O^{-1}\left(\mathbb{R}^3\right)$ means $u=\operatorname{div} \boldsymbol{V}$ for some anti-symmetric tensor $\boldsymbol{V} \in B M O\left(\mathbb{R}^3\right)$. The $B M O^{-1}$ condition was later relaxed to growth conditions on a mean oscillation of $\boldsymbol{V}$ in \cites{S2018,CW2019,BY2024}. For more references on conditional Liouville properties or Liouville properties on special domains, we refer to \cites{bang2022liouville,CPZZ2020,Chae2021,KTW2021,LRZ2022, SW2019} and the references therein.
	Relatively speaking, the two-dimensional case is easier because the vorticity of the 2D Navier-Stokes equations satisfies a nice elliptic equation, to which the maximum principle applies. For example, Gilbarg-Weinberger \cite{GW1978} obtained Liouville type theorem provided the Dirichlet energy is finite. As a different type of Liouville property for the 2D Navier-Stokes equations, Koch-Nadirashvili-Seregin-Sverak \cite{KNSS2009} showed that any bounded solution is a trivial solution as a byproduct of their results on the non-stationary case (see also \cite{BFZ2013} for the unbounded velocity).
	
	However, for the MHD system, the situation is quite different. Due to the lack of maximum principle, there is not much progress in the study of MHD equation. For the 2D MHD equations, Liouville type theorems were proved by assuming the smallness of the norm of the magnetic field in \cite{WW2019}, and De Nitti et al. \cite{DHS2022} removed the smallness assumption. For the 3D MHD equations, Chae-Weng \cite{CW2016} proved that if the smooth solution satisfies D-condition ($\int_{\mathbb{R}^3}\left(|\nabla u|^2+|\nabla B|^2\right) d x<\infty$) and $u \in L^3(\mathbb{R}^3)$, then the solutions $(u, B)$ are identically zero. In \cite{S2019}, Schulz proved that if the smooth solution $(u, B)$ of the stationary MHD equations is in $L^6(\mathbb{R}^3)$ and $u,B \in BMO^{-1}(\mathbb{R}^3)$ then it is identically zero. Chae-Wolf \cite{CW2021} showed that $L^6$ mean oscillations of the potential function of the velocity and magnetic field have certain linear growth by using the technique of \cite{CW2019}. For more references, we refer to \cites{LP2021,LLN2020,FW2021,WY2024} and the references therein. 
	
	Here are our main results.
	\begin{Thm} \label{thm:main}
		Let ${u}$ and $B$ be smooth solutions of \eqref{eq:MHD}. Assume that there exists an $s \in(3,6]$ and smooth anti-symmetric potentials $\boldsymbol{V},\boldsymbol{W} \in C^{\infty}\left(\mathbb{R}^3 ; \mathbb{R}^{3 \times 3}\right)$ such that $\nabla \cdot \boldsymbol{V}=u$, $\nabla \cdot \boldsymbol{W}=B$, and
		\begin{equation}\begin{aligned} \label{eq:main condition}
				\left\|\boldsymbol{V}-(\boldsymbol{V})_{B_R}\right\|_{L^s(B_R)} + \left\|\boldsymbol{W}-(\boldsymbol{W})_{B_R}\right\|_{L^s(B_R)} \lesssim R^{\frac{s+6}{3 s}}(\log R)^{\beta},
		\end{aligned}\end{equation}
		where $\beta = \frac{s+6}{6 s}$, for all $R>2$.
		Then $u=B \equiv 0$.
	\end{Thm}

	\begin{corollary} \label{cor1}
		Let ${u}$ be a smooth solution of \eqref{ns}. Assume that there exists an $s \in(3,6]$ and smooth anti-symmetric potential $\boldsymbol{V}\in C^{\infty}\left(\mathbb{R}^3 ; \mathbb{R}^{3 \times 3}\right)$ such that $\nabla \cdot \boldsymbol{V}=u$, and
		\begin{equation*}\begin{aligned} \label{eq:con}
				\left\|\boldsymbol{V}-(\boldsymbol{V})_{B_R}\right\|_{L^s(B_R)} \lesssim R^{\frac{s+6}{3 s}}(\log R)^{\frac{s+6}{6 s}},
		\end{aligned}\end{equation*}
		for all $R>2$.
		Then $u \equiv 0$.
	\end{corollary}
	\begin{Rem}
		We emphasize that for any divergence-free vector field \( u \), there always exists an anti-symmetric potential \( \boldsymbol{V} \) such that \( u = \operatorname{div} \boldsymbol{V} \). See, e.g., \cites{Seregin2016, CW2019, CW2021}.
	\end{Rem}
	\begin{Rem}
		In Theorem \ref{thm:main}, $s \in(3,6]$, but it can be generalized to $s \in(3, \infty)$ by replacing \eqref {eq:main condition} with
		$$
		\left\|\boldsymbol{V}-(\boldsymbol{V})_{B_R}\right\|_{L^s(B_R)}  + \left\|\boldsymbol{W}-(\boldsymbol{W})_{B_R}\right\|_{L^s(B_R)}\lesssim R^{\min \left\{\frac{s+6}{3 s}, \frac{s+18}{6 s}\right\}}(\log R)^{\min \left\{\frac{s+6}{6 s}, \frac{1}{3}\right\}},
		$$
		for all $R>2$. Indeed, if $s>6$, by Hölder inequality,
		\begin{equation*}\begin{aligned}
				&\quad\left\|\boldsymbol{V}-(\boldsymbol{V})_{B_R}\right\|_{L^6(B_R)}+ \left\|\boldsymbol{W}-(\boldsymbol{W})_{B_R}\right\|_{L^6(B_R)} \\ &\lesssim R^{\frac{s-6}{2 s}} \left(\left\|\boldsymbol{V}-(\boldsymbol{V})_{B_R}\right\|_{L^s(B_R)}  + \left\|\boldsymbol{W}-(\boldsymbol{W})_{B_R}\right\|_{L^s(B_R)}\right) \lesssim R^{\frac{2}{3}}(\log R)^{\frac{1}{3}},
		\end{aligned}\end{equation*}
	    which implies \eqref{eq:main condition} when $s=6$.
	\end{Rem}
	\begin{Rem}
		For \( s = 3 \), we can not control the local Dirichlet energy (see \eqref{eq:s>3} for details). However, if we assume \( \int_{\mathbb{R}^3} |\nabla u|^2 + |\nabla B|^2~\mathrm{d}x < \infty \), we could set \( s = 3 \). In fact, we only consider Case I in the proof of Theorem \ref{thm:main}, and the proof by contradiction still holds due to we do not need the growth estimates for the local Dirichlet energy.
	\end{Rem}
	\begin{Rem}
		The $\log R$ factor in \eqref{eq:main condition} and \eqref{eq:con} can be  relaxed to $\prod_{k=1}^m l^k(R)$ for any $m \in \mathbb{N}$, where $l^k$ is the composition of $|\log |$ function $k$ times. Indeed, the analytical method of the Froullani integral remains valid for $\prod_{k=1}^m l^k(R)$.
	\end{Rem}
	\begin{Rem}
		Theorem \ref{thm:main} relaxes the assumption in Theorem 1.1 of \cite{CW2021} from \( R^{\frac{s+6}{3s}} \) to \( R^{\frac{s+6}{3s}} (\log R)^{\frac{s+6}{6s}} \). Also, Corollary \ref{cor1} relaxes the assumption in Theorem 1.1 of \cite{BY2024} from \( R^{\frac{s+6}{3s}} (\log R)^{\frac{s-3}{3s}} \) to \( R^{\frac{s+6}{3s}} (\log R)^{\frac{s+6}{6s}} \). We improved the coefficient of \( \log R \) by refining the Saint-Venant type estimates using the Froullani integral. See \eqref{eq:main int}–\eqref{eq:Ib} for more details. We also note that improving the coefficient of \( R \) is difficult due to limitations of the method.
	\end{Rem}
	This paper is organized as follows: some notations and some necessary lemmas are presented in Sect.\ref{sec2}; The Liouville type theorem on the MHD system  is obtained in Sect.\ref{sec3}, where we give the detailed proof of Theorem \ref{thm:main}.

	\section{Preliminaries} \label{sec2}
	First we introduce some notations. Denote by $\nabla^\gamma:=\partial_{x_1}^{\gamma_1} \partial_{x_2}^{\gamma_2} \partial_{x_3}^{\gamma_3}$, where  $\gamma=\left(\gamma_1, \gamma_2, \gamma_3\right)$, $\gamma_1, \gamma_2, \gamma_3 \in \mathbb{N} \cup\{0\}$, $\partial_i = \frac{\partial}{\partial x_i}$ and $|\gamma|=\gamma_1+\gamma_2+\gamma_3$. We denote $L^p(\Omega)$ by the usual Lebesgue space with the norm
	$$
	\|f\|_{L^p(\Omega)}:= \begin{cases}\left(\int_{\Omega}|f(x)|^p d x\right)^{1 / p}, & 1 \leq p<\infty, \\ \underset{x \in \Omega}{\operatorname{ess\, sup}}|f(x)|, & p=\infty,\end{cases}
	$$
	where $\Omega \subset \mathbb{R}^3, 1 \leq p \leq \infty$. $W^{k, p}(\Omega)$  and $\dot{W}^{k, p}(\Omega)$ are defined as follows:
	$$
	\begin{aligned}
		\|f\|_{W^{k, p}(\Omega)} & :=\sum_{0 \leq|\gamma| \leq k}\left\|\nabla^\gamma f\right\|_{L^p(\Omega)}, \\
		\|f\|_{\dot{W}^{k, p}(\Omega)} & :=\sum_{|\gamma|=k}\left\|\nabla^\gamma f\right\|_{L^p(\Omega)},
	\end{aligned}
	$$
	respectively. Let $C_c^{\infty}(\Omega)$ denote the space of infinitely differentiable functions with compact support in $\Omega$. We denote by $W_0^{k, p}(\Omega) $ the closure of $C_c^{\infty}(\Omega)$ in $W^{k, p}(\Omega)$. Throughout this paper, we denote $\overline{\boldsymbol{V}}:=\boldsymbol{V}-(\boldsymbol{V})_{B_R}$, $\overline{\boldsymbol{W}}:=\boldsymbol{W}-(\boldsymbol{W})_{B_R}$ and denote $A \leq C B$ by $A \lesssim B$ for convenience. We denote that $B_R$ is the ball in $\mathbb{R}^3$ of radius $R$ centered at the origin, $B_R^{+}=B_R \cap \mathbb{R}_{+}^3$. We denote by
	$$
	\left\|\left(f,g\right)\right\|^2_{L^2(\Omega)} := \left\|f\right\|^2_{L^2(\Omega)} + \left\|g\right\|^2_{L^2(\Omega)} .
	$$
	Let $0<\rho<R$, we denote $\eta_1$ by a cut-off function such that 
	$$
	\eta_1(x)=\eta_1(|x|)=\left\{\begin{array}{lll}
		1, & \text { if } & |x|< \frac{R+\rho}{2} ; \\
		0, & \text { if } & |x|> R ,
	\end{array}\right.
	$$
	and $0 \leq \eta_1(x) \leq 1$ for any $\frac{R+\rho}{2}  \leq|x| \leq R$. We denote $\eta_2$ by a cut-off function such that 
	$$
	\eta_2(x)=\eta_2(|x|)=\left\{\begin{array}{lll}
		1, & \text { if } & |x|< \rho ; \\
		0, & \text { if } & |x|>  \frac{R+\rho}{2},
	\end{array}\right.
	$$
	and $0 \leq \eta_2(x) \leq 1$ for any $\rho  \leq|x| \leq \frac{R+\rho}{2}$. We denote $\eta$ by a cut-off function given by
	$$
	\eta(x)= \begin{cases}1 & \text { if } x<1, \\ -x+2 & \text { if } 1 \leq x \leq 2, \\ 0 & \text { if } x>2,\end{cases}
	$$
	and denote $\varphi_R(x):=\eta(\frac{|x|}{R} )$. It obviously holds
	\begin{equation*}\begin{aligned}
			|\nabla \eta_1|  +  |\nabla \eta_2|  \lesssim  \left(R-\rho\right)^{-1}.
	\end{aligned}\end{equation*}

	Next, we apply the following two lemmas to demonstrate that \( \left\| u \eta_1 \right\|_{L^2(B_R)} \) and \( \left\| u \eta_1 \right\|_{L^{\frac{4 s}{s+2}}(B_R)} \) are controlled by \( \|\overline{\boldsymbol{V}}\|_{L^s(B_R)} \) and \( \|\nabla u\|_{L^2(B_R)} \), with the same result holding for $B$.
	
	\begin{Lem}[\cite{BY2024}, Lemma 2.1] \label{lem:est of 2-norm}
		Assume $s \geq 2$. It holds
		\begin{equation}\begin{aligned}
				\left\|{u} \eta_1\right\|_{L^2(B_R)} &\lesssim R^{\frac{3(s-2)}{4 s}}\|\overline{\boldsymbol{V}}\|_{L^s(B_R)}^{\frac{1}{2}}\|\nabla {u}\|_{L^2(B_R)}^{\frac{1}{2}}+R^{\frac{3(s-2)}{2 s}}(R-\rho)^{-1}\|\overline{\boldsymbol{V}}\|_{L^s(B_R)}, \\
					\left\|B \eta_1\right\|_{L^2(B_R)} &\lesssim R^{\frac{3(s-2)}{4 s}}\|\overline{\boldsymbol{W}}\|_{L^s(B_R)}^{\frac{1}{2}}\|\nabla B\|_{L^2(B_R)}^{\frac{1}{2}}+R^{\frac{3(s-2)}{2 s}}(R-\rho)^{-1}\|\overline{\boldsymbol{W}}\|_{L^s(B_R)}.
		\end{aligned}\end{equation}
	\end{Lem}
	\begin{Lem}[\cite{BY2024}, Lemma 2.2 ] \label{lem: est of 4s}
		Assume $s \geq 2$. It holds
		\begin{equation}\begin{aligned}
				\left\|{u} \eta_1\right\|_{L^{\frac{4 s}{s+2}}(B_R)} &\lesssim\|\overline{\boldsymbol{V}}\|_{L^s(B_R)}^{\frac{1}{2}}\|\nabla {u}\|_{L^2(B_R)}^{\frac{1}{2}}+(R-\rho)^{-1} R^{\frac{3 s-6}{4 s}}\|\overline{\boldsymbol{V}}\|_{L^s(B_R)},\\
				\left\|{B} \eta_1\right\|_{L^{\frac{4 s}{s+2}}(B_R)} &\lesssim\|\overline{\boldsymbol{W}}\|_{L^s(B_R)}^{\frac{1}{2}}\|\nabla {B}\|_{L^2(B_R)}^{\frac{1}{2}}+(R-\rho)^{-1} R^{\frac{3 s-6}{4 s}}\|\overline{\boldsymbol{W}}\|_{L^s(B_R)}.
		\end{aligned}\end{equation}
	\end{Lem}
	
	We recall the Bogovskii map (see \cite{Galdi}, Lemma III.3.1 and \cite{T2018}, §2.8). For a domain $\Omega \subset \mathbb{R}^3$, denote
	$$
	L_0^q(\Omega)=\left\{f \in L^q(\Omega): \int_\Omega f~\mathrm{d}x=0\right\}.
	$$
	Then we use the following lemma to deal with the pressure.
	\begin{Lem}[\cite{T2021}, Lemma 3] \label{lem:deal with the pressure}
		Let $R>0$, $1<L<\infty$ and $A=B_{L R} \backslash \bar{B}_R$ or $A=B_{L R}^{+} \backslash \bar{B}_R^{+}$ be an annulus or a half-annulus in $\mathbb{R}^3$. There is a linear Bogovskii map $\operatorname{Bog}$ that maps a scalar function $f \in L_0^q(A), 1<q<\infty$, to a vector field $v=\operatorname{Bog} f \in W_0^{1, q}(A)$ and
		$$
		\operatorname{div} v=f, \quad\|\nabla v\|_{L^q(A)} \leq \frac{C_q}{(L-1) L^{1-1 / q}}\|f\|_{L^q(A)}.
		$$
		The constant $C_q$ is independent of $L$ and $R$.
	\end{Lem}
	
	Using the above three lemmas, we obtain the following energy inequality, which is crucial for proving the main theorem.
	\begin{Lem} \label{lem:mian est}
		Assume $2<\rho<R$, $3\leq s\leq6$ and $\beta>0$. Let $\boldsymbol{V}$ and $\boldsymbol{W}$ satisfy \eqref{eq:main condition}. Then
		\begin{eqnarray*}
			&&\left\|\left(\nabla u, \nabla B\right)\sqrt{\eta_2}\right\|_{L^2(B_R)}^2 \\&\lesssim& R^{10}(R-\rho)^{-10}(\log R)^{\frac{3s}{s+6}\beta}\|\left(\nabla u, \nabla B\right)\|_{L^2(B_{\frac{\rho+R}{2}}\setminus B_\rho)}\left(\|\left(\nabla u, \nabla B\right)\|_{L^2(B_R)}^{\frac{12-s}{s+6}}+1\right) . 
		\end{eqnarray*}
	\end{Lem}
	\begin{proof}
		We divide the proof into two steps.\\
		\underline{\bf Step I. Dealing with the pressure.}
		 Using Lemma \ref{lem:deal with the pressure} with $L = \frac{R+\rho}{2\rho}$, for any $q \in(1, \infty)$, there exists a constant $C_q>0$ and a function ${w} \in W_0^{1, q}\left(B_{(R+\rho) / 2} \backslash B_\rho\right)$ such that $\operatorname{div} {w}={u} \cdot \nabla \eta_2$ and
		\begin{equation}\begin{aligned} \label{eq:est of w}
			\int_{B_{(R+\rho) / 2} \backslash B_\rho}|\nabla {w}|^q ~\mathrm{d}x \leq C_q\left(\frac{\rho}{R-\rho}\right)^q \int_{B_{(R+\rho) / 2} \backslash B_\rho}\left|{u} \cdot \nabla \eta_2\right|^q ~\mathrm{d}x.
		\end{aligned}\end{equation}
		We then extend ${w}=0$ in $B_\rho$. Integrating by part, we find
		\begin{equation}\begin{aligned} \label{treat p}
				\int_{B_R} \nabla p \cdot \left(u\eta_2 - w\right) ~\mathrm{d}x = -\int_{B_R} pu\cdot\nabla\eta_2 ~\mathrm{d}x  + \int_{B_R} p \nabla\cdot w ~\mathrm{d}x = 0.
		\end{aligned}\end{equation}
		
		\underline{\bf Step II. $L^2$-estimate.}
		Multiplying the equation $\eqref{eq:MHD}_1$ by ${u} \eta_2-{w}$, multiplying the equation $\eqref{eq:MHD}_2$ by $B \eta_2$  and integrating by parts in $B_R$, we have
		\begin{equation}\begin{aligned} \label{eq:energy}
				\int_{B_R}|\nabla {u}|^2 \eta_2 +  |\nabla {B}|^2 \eta_2 ~\mathrm{d}x= & -\int_{B_R} u\cdot\nabla u\cdot\nabla\eta_2 + B\cdot\nabla B\cdot\nabla\eta_2 ~\mathrm{d}x   +   \int_{B_R} \nabla {u} \cdot \nabla {w} ~\mathrm{d}x \\
				& + \int_{B_R} u\cdot\nabla u\cdot w - B\cdot\nabla B\cdot w ~\mathrm{d}x  -\int_{B_R} u\cdot\nabla u\cdot u \eta_2  +  u\cdot\nabla B\cdot B \eta_2 ~\mathrm{d}x   \\
				&+ \int_{B_R} B\cdot\nabla B\cdot u\eta_2  + B\cdot\nabla u \cdot B\eta_2 ~\mathrm{d}x\\
				= & : I+I I+I I I+I V + V,
		\end{aligned}\end{equation}
	    where the pressure vanishes due to \eqref{treat p}. For I, by H\"{o}lder inequality and Lemma \ref{lem:est of 2-norm}, it holds
		\begin{equation}\begin{aligned} \label{eq:I}
				|I| & \lesssim(R-\rho)^{-1} \left(\|\nabla u\|_{L^2(B_{(R+\rho) / 2} \backslash B_\rho)}   \| u\|_{L^2(B_{(R+\rho) / 2} \backslash B_\rho)} +  \|\nabla B\|_{L^2(B_{(R+\rho) / 2} \backslash B_\rho)}\| B\|_{L^2(B_{(R+\rho) / 2} \backslash B_\rho)}\right)  \\ 
				& \lesssim(R-\rho)^{-1}   \left(\|\nabla u\|_{L^2(B_{(R+\rho) / 2} \backslash B_\rho)}\| u\eta_1\|_{L^2(B_R)} + \|\nabla B\|_{L^2(B_{(R+\rho) / 2} \backslash B_\rho)}\| B\eta_1\|_{L^2(B_R)}\right) \\ 
				& \lesssim(R-\rho)^{-1}\|\left(\nabla u, \nabla B\right)\|_{L^2(B_{(R+\rho) / 2} \backslash B_\rho)}
				\left[R^{\frac{3(s-2)}{4 s}}  \|\left(\nabla {u}, \nabla B\right)\|_{L^2(B_R)}^{\frac{1}{2}} \left(\|\overline{\boldsymbol{V}}\|_{L^s(B_R)}^{\frac{1}{2}}+\|\overline{\boldsymbol{W}}\|_{L^s(B_R)}^{\frac{1}{2}}\right)  \right.\\
				&\left.\quad+   R^{\frac{3(s-2)}{2 s}}(R-\rho)^{-1}\left(\|\overline{\boldsymbol{V}}\|_{L^s(B_R)} + \|\overline{\boldsymbol{W}}\|_{L^s(B_R)}\right)\right] .
		\end{aligned}\end{equation}
	    For II, by Hölder inequality and \eqref{eq:est of w} with $q=2$, we get
	\begin{equation}\begin{aligned} \label{eq:II}
			|II| & \lesssim(R-\rho)^{-1} \|\nabla u\|_{L^2(B_{(R+\rho) / 2} \backslash B_\rho)}   \| \nabla w\|_{L^2(B_{(R+\rho) / 2} \backslash B_\rho)}  \\ 
			& \lesssim R(R-\rho)^{-2}  \|\nabla u\|_{L^2(B_{(R+\rho) / 2} \backslash B_\rho)}   \| u\|_{L^2(B_{(R+\rho) / 2} \backslash B_\rho)} .
	\end{aligned}\end{equation}
     Combining \eqref{eq:I} and \eqref{eq:II}, we have
     \begin{equation}\begin{aligned} \label{eq:I+}
     		&|I|+|II|
     		\lesssim R(R-\rho)^{-2}\|\left(\nabla u, \nabla B\right)\|_{L^2(B_{(R+\rho) / 2} \backslash B_\rho)}\\
     		&\cdot\left[R^{\frac{3(s-2)}{4 s}}  \|\left(\nabla {u}, \nabla B\right)\|_{L^2(B_R)}^{\frac{1}{2}} \left(\|\overline{\boldsymbol{V}}\|_{L^s(B_R)}^{\frac{1}{2}}+\|\overline{\boldsymbol{W}}\|_{L^s(B_R)}^{\frac{1}{2}}\right) 
     		+   R^{\frac{3(s-2)}{2 s}}(R-\rho)^{-1}\left(\|\overline{\boldsymbol{V}}\|_{L^s(B_R)} + \|\overline{\boldsymbol{W}}\|_{L^s(B_R)}\right)\right] .
     \end{aligned}\end{equation}
	For III, by integration by parts, it holds
	\begin{equation*}\begin{aligned} 
			III &= \int_{B_R} \partial_k \overline{\boldsymbol{V}}_{ki} \partial_i u_j w_j  -  \partial_k \overline{\boldsymbol{W}}_{ki} \partial_i B_j w_j ~\mathrm{d}x \\
			&= -\int_{B_R}  \overline{\boldsymbol{V}}_{ki} \partial_i u_j \partial_k w_j  -   \overline{\boldsymbol{W}}_{ki} \partial_i B_j \partial_kw_j ~\mathrm{d}x,
	\end{aligned}\end{equation*}
	where we use the fact that $\overline{\boldsymbol{V}}$ and $\overline{\boldsymbol{W}}$ are anti-symmetric, $\nabla^2(u\cdot w)$ and $\nabla^2(B\cdot w)$ are symmetric. Then by Hölder inequality and \eqref{eq:est of w} with $q=\frac{2s}{s-2}$, we have
	\begin{equation}\begin{aligned} \label{eq:III}
			|III| &\lesssim \|\nabla w\|_{L^{\frac{2s}{s-2}}(B_{(R+\rho) / 2} \backslash B_\rho)}\|\left(\nabla u, \nabla B\right)\|_{L^2(B_{(R+\rho) / 2} \backslash B_\rho)} \left(\|\overline{\boldsymbol{V}}\|_{L^s(B_R)} + \|\overline{\boldsymbol{W}}\|_{L^s(B_R)}\right) \\
			&\lesssim \rho\left(R-\rho\right)^{-2} \left(\|\overline{\boldsymbol{V}}\|_{L^s(B_R)} + \|\overline{\boldsymbol{W}}\|_{L^s(B_R)}\right) \|\left(\nabla u, \nabla B\right)\|_{L^2(B_{(R+\rho) / 2} \backslash B_\rho)}  \|u\eta_1^2\|_{L^{\frac{2s}{s-2}}(B_R)}.
	\end{aligned}\end{equation}
	For IV, by integration by parts, it holds
	\begin{equation}\begin{aligned}
			IV &= -\frac12\int_{B_R} \partial_k \overline{\boldsymbol{V}}_{ki} \partial_i |u|^2 \eta_2  -  \partial_k \overline{\boldsymbol{V}}_{ki} \partial_i |B|^2 \eta_2 ~\mathrm{d}x \\
			&= \frac12\int_{B_R}  \overline{\boldsymbol{V}}_{ki} \partial_i |u|^2 \partial_k \eta_2 -   \overline{\boldsymbol{V}}_{ki} \partial_i |B|^2 \partial_k \eta_2 ~\mathrm{d}x,
	\end{aligned}\end{equation}
	where we use the fact that $\nabla^2 |u|^2$ and $\nabla^2 |B|^2
	$ are symmetric. Then by Hölder inequality, we have
	\begin{equation}\begin{aligned} \label{eq:IV}
			|IV| \lesssim \left(R-\rho\right)^{-1} \|\overline{\boldsymbol{V}}\|_{L^s(B_R)}  \|\left(\nabla u, \nabla B\right)\|_{L^2(B_{(R+\rho) / 2} \backslash B_\rho)}  \left(\|u\eta_1^2\|_{L^{\frac{2s}{s-2}}(B_R)}+ \|B\eta_1^2\|_{L^{\frac{2s}{s-2}}(B_R)}\right).
	\end{aligned}\end{equation}
     For V, by integration by parts, it holds
     \begin{equation*}\begin{aligned}
     		V &= \int_{B_R} \partial_k \overline{\boldsymbol{W}}_{ki} \partial_i B_j u_j \eta_2  +  \partial_k \overline{\boldsymbol{W}}_{ki} \partial_i u_j B_j \eta_2 ~\mathrm{d}x \\
     		&= -\int_{B_R}   \overline{\boldsymbol{W}}_{ki} \partial_i B_j  u_j  \partial_k\eta_2   + \overline{\boldsymbol{W}}_{ki} \partial_i u_j  B_j \partial_k \eta_2 ~\mathrm{d}x ,
     \end{aligned}\end{equation*}
 where we use the fact that $\left[\partial_i B_j \partial_k u_j + \partial_i u_j \partial_k B_j\right]$
  is symmetric. Then by Hölder inequality, we have
  \begin{equation}\begin{aligned} \label{eq:V}
  		|V| \lesssim \left(R-\rho\right)^{-1} \|\overline{\boldsymbol{W}}\|_{L^s(B_R)}  \|\left(\nabla u, \nabla B\right)\|_{L^2(B_{(R+\rho) / 2} \backslash B_\rho)}  \left(\|u\eta_1^2\|_{L^{\frac{2s}{s-2}}(B_R)}+ \|B\eta_1^2\|_{L^{\frac{2s}{s-2}}(B_R)}\right).
  \end{aligned}\end{equation}
	Combining \eqref{eq:III}, \eqref{eq:IV} and \eqref{eq:V}, we get 
	\begin{equation}\begin{aligned}
			&|III|+|IV| + |V|\\
			&\lesssim R\left(R-\rho\right)^{-2} \left(\|\overline{\boldsymbol{V}}\|_{L^s(B_R)} + \|\overline{\boldsymbol{W}}\|_{L^s(B_R)}\right) \|\left(\nabla u, \nabla B\right)\|_{L^2(B_{(R+\rho) / 2} \backslash B_\rho)}  \left(\|u\eta_1^2\|_{L^{\frac{2s}{s-2}}(B_R)}  + \|B\eta_1^2\|_{L^{\frac{2s}{s-2}}(B_R)}\right).
	\end{aligned}\end{equation}
	By interpolation inequality and Sobolev inequality, we get
	\begin{equation}\begin{aligned} \label{eq:u eta2}
			\|u\eta_1^2\|_{L^{\frac{2s}{s-2}}(B_R)} &\leq \|u\eta_1^2\|^{\frac{4s-12}{s+6}}_{L^{\frac{4s}{s+2}}(B_R)} \|u\eta_1^2\|^{\frac{18-3s}{s+6}}_{L^{6}(B_R)} \\
			&\lesssim \|u\eta_1^2\|^{\frac{4s-12}{s+6}}_{L^{\frac{4s}{s+2}}(B_R)} \|\nabla(u\eta_1^2)\|^{\frac{18-3s}{s+6}}_{L^{2}(B_R)} \\
			&\lesssim \|u\eta_1^2\|^{\frac{4s-12}{s+6}}_{L^{\frac{4s}{s+2}}(B_R)}  \left(\|\nabla u\|_{L^2(B_R)}^{\frac{18-3s}{s+6}}  +   (R-\rho)^{-\frac{18-3s}{s+6}}  \| u \eta_1\|_{L^2(B_R)}^{\frac{18-3s}{s+6}}  \right).
	\end{aligned}\end{equation}
	Similarly, we get
	\begin{equation}\begin{aligned} \label{eq:B eta2}
				\|B\eta_1^2\|_{L^{\frac{2s}{s-2}}(B_R)} \lesssim \|B\eta_1^2\|^{\frac{4s-12}{s+6}}_{L^{\frac{4s}{s+2}}(B_R)}  \left(\|\nabla B\|_{L^2(B_R)}^{\frac{18-3s}{s+6}}  +   (R-\rho)^{-\frac{18-3s}{s+6}}  \| B \eta_1\|_{L^2(B_R)}^{\frac{18-3s}{s+6}}  \right).
	\end{aligned}\end{equation}
	Combining \eqref{eq:u eta2} and \eqref{eq:B eta2}, using Lemma \ref{lem:est of 2-norm} and Lemma \ref{lem: est of 4s}, we have
	\begin{eqnarray*}
		&&\|u\eta_1^2\|_{L^{\frac{2s}{s-2}}(B_R)} + \|B\eta_1^2\|_{L^{\frac{2s}{s-2}}(B_R)}  \\ &\lesssim&\left(\|\overline{\boldsymbol{V}}\|_{L^s(B_R)} ^{\frac{2 s-6}{s+6}}+\|\overline{\boldsymbol{W}}\|_{L^s(B_R)}^{\frac{2 s-6}{s+6}}\right)\|\left(\nabla u,\nabla B\right)\|_{L^2(B_R)}^{\frac{12-s}{s+6}} \\
		&& +(R-\rho)^{-\frac{18-3 s}{s+6}} R^{\frac{9(s-2)(6-s)}{4 s(s+6)}}\left(\|\overline{\boldsymbol{V}}\|_{L^s(B_R)}^{\frac{1}{2}}+\|\overline{\boldsymbol{W}}\|_{L^s(B_R)}^{\frac{1}{2}}\right)\|\left(\nabla u,\nabla B\right)\|_{L^2(B_R)}^{\frac{1}{2}} \\
		&& +(R-\rho)^{-\frac{6(6-s)}{s+6}} R^{\frac{9(s-2)(6-s)}{2 s(s+6)}}\left(\|\overline{\boldsymbol{V}}\|_{L^s(B_R)}^{\frac{12-s}{s+6}}+\|\overline{\boldsymbol{W}}\|_{L^s(B_R)}^{\frac{12-s}{s+6}}\right)\|\left(\nabla u,\nabla B\right)\|_{L^2(B_R)}^{\frac{2 s-6}{s+6}} \\
		&& +(R-\rho)^{-\frac{4 s-12}{s+6}} R^{\frac{3(s-2)(s-3)}{s(s+6)}}\left(\|\overline{\boldsymbol{V}}\|_{L^s(B_R)}^{\frac{4s-12}{s+6}}+\|\overline{\boldsymbol{W}}\|_{L^s(B_R)}^{\frac{4s-12}{s+6}}\right)\|\left(\nabla u,\nabla B\right)\|_{L^2(B_R)}^{\frac{18-3 s}{s+6}} \\
		&& +(R-\rho)^{-1} R^{\frac{3(s-2)}{4 s}}\left(\|\overline{\boldsymbol{V}}\|_{L^s(B_R)}^{\frac{5 s-6}{2(s+6)}}+\|\overline{\boldsymbol{W}}\|_{L^s(B_R)}^{\frac{5 s-6}{2(s+6)}}\right)\|\left(\nabla u,\nabla B\right)\|_{L^2(B_R)}^{\frac{18-3 s}{2(s+6)}} \\
		&& +(R-\rho)^{-\frac{2(12-s)}{s+6}} R^{\frac{3(12-s)(s-2)}{2 s(s+6)}}\left(\|\overline{\boldsymbol{V}}\|_{L^s(B_R)}+\|\overline{\boldsymbol{W}}\|_{L^s(B_R)}\right). 
	\end{eqnarray*}
    Hence, we get
    \begin{eqnarray}
    	\label{eq:III+}
    	&&|III|+|IV| + |V| \no\\
    	&\lesssim& R\left(R-\rho\right)^{-2}  \|\left(\nabla u, \nabla B\right)\|_{L^2(B_{(R+\rho) / 2} \backslash B_\rho)}  \left[\left(\|\overline{\boldsymbol{V}}\|_{L^s(B_R)} ^{\frac{3s}{s+6}}+\|\overline{\boldsymbol{W}}\|_{L^s(B_R)}^{\frac{3s}{s+6}}\right)\|\left(\nabla u,\nabla B\right)\|_{L^2(B_R)}^{\frac{12-s}{s+6}} \right. \no\\
    	&&\left. +(R-\rho)^{-\frac{18-3 s}{s+6}} R^{\frac{9(s-2)(6-s)}{4 s(s+6)}}\left(\|\overline{\boldsymbol{V}}\|_{L^s(B_R)}^{\frac{3}{2}}+\|\overline{\boldsymbol{W}}\|_{L^s(B_R)}^{\frac{3}{2}}\right)\|\left(\nabla u,\nabla B\right)\|_{L^2(B_R)}^{\frac{1}{2}} \right.\no\\
    	&&\left. +(R-\rho)^{-\frac{6(6-s)}{s+6}} R^{\frac{9(s-2)(6-s)}{2 s(s+6)}}\left(\|\overline{\boldsymbol{V}}\|_{L^s(B_R)}^{\frac{18}{s+6}}+\|\overline{\boldsymbol{W}}\|_{L^s(B_R)}^{\frac{18}{s+6}}\right)\|\left(\nabla u,\nabla B\right)\|_{L^2(B_R)}^{\frac{2 s-6}{s+6}} \right.\no\\
    	&&\left. +(R-\rho)^{-\frac{4 s-12}{s+6}} R^{\frac{3(s-2)(s-3)}{s(s+6)}}\left(\|\overline{\boldsymbol{V}}\|_{L^s(B_R)}^{\frac{5s-6}{s+6}}+\|\overline{\boldsymbol{W}}\|_{L^s(B_R)}^{\frac{5s-6}{s+6}}\right)\|\left(\nabla u,\nabla B\right)\|_{L^2(B_R)}^{\frac{18-3 s}{s+6}} \right.\no\\
    	&&\left. +(R-\rho)^{-1} R^{\frac{3(s-2)}{4 s}}\left(\|\overline{\boldsymbol{V}}\|_{L^s(B_R)}^{\frac{7 s+6}{2(s+6)}}+\|\overline{\boldsymbol{W}}\|_{L^s(B_R)}^{\frac{7 s+6}{2(s+6)}}\right)\|\left(\nabla u,\nabla B\right)\|_{L^2(B_R)}^{\frac{18-3 s}{2(s+6)}} \right.\no\\
    	&&\left. +(R-\rho)^{-\frac{2(12-s)}{s+6}} R^{\frac{3(12-s)(s-2)}{2 s(s+6)}}\left(\|\overline{\boldsymbol{V}}\|_{L^s(B_R)}^2+\|\overline{\boldsymbol{W}}\|_{L^s(B_R)}^2\right)\right].
    \end{eqnarray}
	Putting \eqref{eq:energy}, \eqref{eq:I+} and \eqref{eq:III+} together, we get
	\begin{eqnarray} \label{eq:all I}
		&&\int_{B_R}|\nabla {u}|^2 \eta_2 +  |\nabla {B}|^2 \eta_2 ~\mathrm{d}x \no\\
		&\lesssim& R\left(R-\rho\right)^{-2}  \|\left(\nabla u, \nabla B\right)\|_{L^2(B_{(R+\rho) / 2} \backslash B_\rho)}     \left[ \left(\|\overline{\boldsymbol{V}}\|_{L^s(B_R)} ^{\frac{3s}{s+6}}+\|\overline{\boldsymbol{W}}\|_{L^s(B_R)}^{\frac{3s}{s+6}}\right)\|\left(\nabla u,\nabla B\right)\|_{L^2(B_R)}^{\frac{12-s}{s+6}} \right. \no\\
		&&\left. +(R-\rho)^{-\frac{18-3 s}{s+6}} R^{\frac{9(s-2)(6-s)}{4 s(s+6)}}\left(\|\overline{\boldsymbol{V}}\|_{L^s(B_R)}^{\frac{3}{2}}+\|\overline{\boldsymbol{W}}\|_{L^s(B_R)}^{\frac{3}{2}}\right)\|\left(\nabla u,\nabla B\right)\|_{L^2(B_R)}^{\frac{1}{2}} \right.\no\\
		&&\left. +(R-\rho)^{-\frac{6(6-s)}{s+6}} R^{\frac{9(s-2)(6-s)}{2 s(s+6)}}\left(\|\overline{\boldsymbol{V}}\|_{L^s(B_R)}^{\frac{18}{s+6}}+\|\overline{\boldsymbol{W}}\|_{L^s(B_R)}^{\frac{18}{s+6}}\right)\|\left(\nabla u,\nabla B\right)\|_{L^2(B_R)}^{\frac{2 s-6}{s+6}} \right.\no\\
		&&\left. +(R-\rho)^{-\frac{4 s-12}{s+6}} R^{\frac{3(s-2)(s-3)}{s(s+6)}}\left(\|\overline{\boldsymbol{V}}\|_{L^s(B_R)}^{\frac{5s-6}{s+6}}+\|\overline{\boldsymbol{W}}\|_{L^s(B_R)}^{\frac{5s-6}{s+6}}\right)\|\left(\nabla u,\nabla B\right)\|_{L^2(B_R)}^{\frac{18-3 s}{s+6}} \right.\no\\
		&&\left. +(R-\rho)^{-1} R^{\frac{3(s-2)}{4 s}}\left(\|\overline{\boldsymbol{V}}\|_{L^s(B_R)}^{\frac{7 s+6}{2(s+6)}}+\|\overline{\boldsymbol{W}}\|_{L^s(B_R)}^{\frac{7 s+6}{2(s+6)}}\right)\|\left(\nabla u,\nabla B\right)\|_{L^2(B_R)}^{\frac{18-3 s}{2(s+6)}} \right.\no\\
		&&\left. +(R-\rho)^{-\frac{2(12-s)}{s+6}} R^{\frac{3(12-s)(s-2)}{2 s(s+6)}}\left(\|\overline{\boldsymbol{V}}\|_{L^s(B_R)}^2+\|\overline{\boldsymbol{W}}\|_{L^s(B_R)}^2\right) \right.\no\\
		&&\left. + R^{\frac{3(s-2)}{4 s}}  \|\left(\nabla {u}, \nabla B\right)\|_{L^2(B_R)}^{\frac{1}{2}} \left(\|\overline{\boldsymbol{V}}\|_{L^s(B_R)}^{\frac{1}{2}}+\|\overline{\boldsymbol{W}}\|_{L^s(B_R)}^{\frac{1}{2}}\right)  \right.\no\\
		&&\left.+   R^{\frac{3(s-2)}{2 s}}(R-\rho)^{-1}\left(\|\overline{\boldsymbol{V}}\|_{L^s(B_R)} + \|\overline{\boldsymbol{W}}\|_{L^s(B_R)}\right)\right]\no\\
		&:=&R\left(R-\rho\right)^{-2}  \|\left(\nabla u, \nabla B\right)\|_{L^2(B_{(R+\rho) / 2} \backslash B_\rho)} \sum_{i=1}^{8} I_i.
	\end{eqnarray}
	We point out that $I_1$ is the main term, since
	\begin{equation}\begin{aligned} \label{eq:I_1}
			I_1 \lesssim  R(\log R)^{\frac{3s}{s+6}\beta}\|\left(\nabla u,\nabla B\right)\|_{L^2(B_R)}^{\frac{12-s}{s+6}},
	\end{aligned}\end{equation}
    by \eqref{eq:main condition}.
    For $I_2$, by \eqref{eq:main condition}, we have
    \begin{equation}\begin{aligned} \label{eq:I_2}
    		I_2 &\lesssim R^{\frac{5 s-6}{4 s}}\left(\frac{R}{R-\rho}\right)^{\frac{18-3 s}{s+6}}(\log R)^{\frac32 \beta}\|\left(\nabla u,\nabla B\right)\|_{2, R}^{\frac{1}{2}} \\
    		&\lesssim R\left(\frac{R}{R-\rho}\right)^8 (\log R)^{\frac{3s}{s+6}\beta} \max\left\{\|\left(\nabla u,\nabla B\right)\|_{L^2(B_R)}^{\frac{12-s}{s+6}}, 1\right\}.
    \end{aligned}\end{equation}
    Similarly, we have
    \begin{equation*}\begin{aligned}
    		\sum_{i=3}^{8} I_i  \lesssim  R\left(\frac{R}{R-\rho}\right)^8 (\log R)^{\frac{3s}{s+6}\beta} \max\left\{\|\left(\nabla u,\nabla B\right)\|_{L^2(B_R)}^{\frac{12-s}{s+6}}, 1\right\}.
    \end{aligned}\end{equation*}
    which, combined with \eqref{eq:all I}, \eqref{eq:I_1} and \eqref{eq:I_2},  completes the proof.
\end{proof}

    With the lemma above, we obtain a growth estimate for the local Dirichlet energy via an iterative argument.
    
    \begin{Lem}[Growth estimates of the local Dirichlet energy] \label{lem:grow est}
    	Assuming that $\boldsymbol{V}$ and $\boldsymbol{W}$  satisfy \eqref{eq:main condition} for $3<s\leq6$, it holds
    	$$
    	\int_{B_R}|\nabla u|^2 + |\nabla B|^2 ~\mathrm{d}x \lesssim (\log R)^{\frac{3s}{s-3}\beta},
    	$$
    	for all $R>2$.
    \end{Lem}
    \begin{proof}
    	We may assume that $\|\left(\nabla u, \nabla B\right)\|_{L^2(B_R)}>1$ when $R>R'$ for some positive $R'$. Fix $R'<\rho<R$. By Lemma \ref{lem:mian est} and Young inequality, we have
    	\begin{equation}\begin{aligned} \label{eq:s>3}
    			\|\left(\nabla u, \nabla B\right)\|_{L^2(B_\rho)}^2 & \lesssim R^{10}(R-\rho)^{-10}(\log R)^{{\frac{3s}{s+6}\beta}}\|\left(\nabla u, \nabla B\right)\|_{L^2(B_R)}^{\frac{18}{s+6}} \\
    			& \leq\frac{1}{2}\|\left(\nabla u, \nabla B\right)\|_{L^2(B_R)}^2+C R^{\frac{10(s+6)}{s-3}}(R-\rho)^{-\frac{10(s+6)}{s-3}} (\log R)^{\frac{3s}{s-3}\beta},
    	\end{aligned}\end{equation}
    	where we note that $s>3$ ensures the validity of Young inequality. Denote $t_0=\rho$, $t_{i+1}-t_i=2^{-i-1}(R-\rho)$. Using above inequality with $\rho, R$ replaced by $t_i, t_{i+1}$, and iterating $k$ times, we have
    	$$
    	\begin{aligned}
    		\|\left(\nabla u, \nabla B\right)\|_{L^2(B_{t_0})}^2 & \leq\frac{1}{2^k}\|\left(\nabla u, \nabla B\right)\|_{L^2(B_{t_k})}^2+C  R^{\frac{10(s+6)}{s-3}}(R-\rho)^{-\frac{10(s+6)}{s-3}} (\log R)^{\frac{3s}{s-3}\beta} \sum_{i=0}^{k-1}\frac{1}{2^i}.
    	\end{aligned}
    	$$
    	Letting $k \rightarrow \infty$ and choosing $\rho=R / 2$, we complete the proof.
    \end{proof}

	\section{Proof of Theorem \ref{thm:main}} \label{sec3}
	
	\begin{proof}[Proof of Theorem \ref{thm:main}]
		We define
		$$
		g(R):=\int_{\mathbb{R}^3} \left(|\nabla u|^2 + |\nabla B|^2\right) \varphi_R ~\mathrm{d}x,
		$$
		which implies
		$$
		g^{\prime}(R)=\int_{B_{2 R} \backslash B_R}\left(|\nabla u|^2 + |\nabla B|^2\right) \eta^{\prime}\left(\frac{|x|}{R}\right) \cdot\left(-\frac{|x|}{R^2}\right) ~\mathrm{d}x \gtrsim R^{-1} \int_{B_{2 R} \backslash B_R}\left(|\nabla u|^2 + |\nabla B|^2\right)~\mathrm{d}x.
		$$
		Replacing $\rho, R$ by $R, 3 R$ and choosing $\eta_2=\varphi_R$ in Lemma \ref{lem:mian est}, we have
		\begin{equation}\begin{aligned}
				\label{eq:est of gr}
				g(R) \lesssim R^{\frac{1}{2}}(\log R)^{\frac{3s}{s+6}\beta} g^{\prime}(R)^{\frac{1}{2}}\left(g(3 R)^{\frac{12-s}{2(s+6)}}+1\right).
		\end{aligned}\end{equation}
		Recall that 
		\begin{equation*}\begin{aligned} \label{eq:est of nabla u}
				\int_{B_R}|\nabla u|^2 + |\nabla B|^2 ~\mathrm{d}x \lesssim (\log R)^{\frac{3s}{s-3}\beta},
		\end{aligned}\end{equation*}
	    for all $R>2$ in Lemma \ref{lem:grow est}. Then by the definition of $g(R)$, we have
		\begin{equation}\begin{aligned} \label{eq:est of gR}
				g(R) \lesssim (\log R)^{\frac{3s}{s-3}\beta},
		\end{aligned}\end{equation}
		for all $R>2$.
		We then prove $\int_{\mathbb{R}^3} |\nabla u|^2 + |\nabla B|^2~\mathrm{d}x = 0$ by contradiction. Assuming $\int_{\mathbb{R}^3} |\nabla u|^2 + |\nabla B|^2~\mathrm{d}x\neq0$, we divide the proof into two cases.
		
		\underline{\bf Case I. $0\neq\int_{\mathbb{R}^3} |\nabla u|^2 + |\nabla B|^2~\mathrm{d}x\leq2$}. There exists $R_0>e$, such that $g(R)>0$ for all $R \geq R_0$. Then by \eqref{eq:est of gr}, choosing $\beta = \frac{s+6}{6 s}$, we directly obtain
		\begin{equation*}\begin{aligned}
				g(R) \lesssim R^{\frac{1}{2}}(\log R)^{\frac12} g^{\prime}(R)^{\frac{1}{2}},
		\end{aligned}\end{equation*}
		which implies
		\begin{equation*}\begin{aligned}
				\int_{R_0}^{+\infty} \frac{~\mathrm{d}R}{R\log R}  \lesssim \int_{R_0}^{+\infty} \frac{g'(R)}{g(R)^2} ~\mathrm{d}R.
		\end{aligned}\end{equation*}
		We derive a contradiction from the divergence of the left-hand side and the convergence of the right-hand side of the above inequality.
		
		\underline{\bf Case II. $\int_{\mathbb{R}^3} |\nabla u|^2 + |\nabla B|^2 ~\mathrm{d}x>2$}. There exists $R_1>9$, such that $g(3R)\geq g(R)>1$ for all $R > R_1$. Then by \eqref{eq:est of gr}, choosing $\beta = \frac{s+6}{6 s}$, we obtain
		\begin{equation*}\begin{aligned}
				g(R) \lesssim R^{\frac{1}{2}}(\log R)^{\frac12} g^{\prime}(R)^{\frac{1}{2}}g(3 R)^{\frac{12-s}{2(s+6)}},
		\end{aligned}\end{equation*}
		which implies
		\begin{equation*}\begin{aligned}
				\frac{\left(\frac{g(R)}{g(3 R)}\right)^{\frac{12-s}{s+6}}}{R\log R}  \lesssim \frac{{g}^{\prime}({R})}{{g}({R})^{\frac{3 {s}}{{s}+6}}}.
		\end{aligned}\end{equation*}
		Integrating both sides on $(R_1, r)$, we get
		\begin{equation}\begin{aligned} \label{eq:main int}
				\int_{R_1}^{r}\frac{\left(\frac{g(R)}{g(3 R)}\right)^{\frac{12-s}{s+6}}}{R\log R}  ~\mathrm{d}R \lesssim \int_{R_1}^{r} \frac{g'(R)}{{g}({R})^{\frac{3 {s}}{{s}+6}}} ~\mathrm{d}R.
		\end{aligned}\end{equation}
		For the left side of \eqref{eq:main int}, we note that
		\begin{equation}\begin{aligned} \label{eq:leftside}
				&\quad\int_{R_1}^{r}\frac{\left(\frac{g(R)}{g(3 R)}\right)^{\frac{12-s}{s+6}}}{R\log R}  ~\mathrm{d}R 
				\gtrsim \int_{R_1}^{r}\frac{\log\frac{g(R)}{g(3R)}+1}{R\log R}  ~\mathrm{d}R \\
				&=\int_{R_1}^{r} \frac{\log g(R)}{R\log R} ~\mathrm{d}R  - \int_{3R_1}^{3r} \frac{\log g(R)}{R\left(\log R - \log3\right)} ~\mathrm{d}R  + \log\log r - \log\log R_1 \\
				&=\int_{R_1}^{r} \frac{\log g(R)}{R}\left(\frac{1}{\log R} - \frac{1}{\log R - \log3}\right) ~\mathrm{d}R -\int_{r}^{3r} \frac{\log g(R)}{R\left(\log R -\log3\right)}~\mathrm{d}R \\
				&\quad+ \int_{R_1}^{3R_1} \frac{\log g(R)}{R\left(\log R-\log3\right)} ~\mathrm{d}R   - \log\log R_1 + \log\log r\\
				&= I_a + I_b +I_c + \log\log r,
		\end{aligned}\end{equation}
		where $$I_a := \int_{R_1}^{r} \frac{\log g(R)}{R}\left(\frac{1}{\log R} - \frac{1}{\log R - \log3}\right) ~\mathrm{d}R,\quad I_b:=-\int_{r}^{3r} \frac{\log g(R)}{R\left(\log R -\log3\right)}~\mathrm{d}R $$ and $$I_c:=\int_{R_1}^{3R_1} \frac{\log g(R)}{R\left(\log R-\log3\right)} ~\mathrm{d}R   - \log\log R_1.$$
		Noting $R_1>9$ implies $\log R - \log3 >\frac12 \log R$ for all $R>R_1$ and using \eqref{eq:est of gR}, we have for $I_a$:
		\begin{equation}\begin{aligned} \label{eq:Ia}
				|I_a| \lesssim \int_{R_1}^{r} \frac{\log g(R)}{R(\log R)^2} ~\mathrm{d}R \lesssim \int_{R_1}^{r} \frac{\log \log R}{R(\log R)^2} ~\mathrm{d}R,
		\end{aligned}\end{equation}
		for $I_b$:
		\begin{equation}\begin{aligned} \label{eq:Ib}
				|I_b| &\lesssim \int_{r}^{3r} \frac{\log \log R}{R\log R}  ~\mathrm{d}R = \left.\frac12\left(\log\log R\right)^2 \right|_r^{3r} = \frac12\left(\log\log 3r\right)^2 - \frac12\left(\log\log r\right)^2 \\
				&= \frac12\left(\log\log r + \log\left(1+\frac{\log3}{\log r}\right)\right)^2 -\frac12\left(\log\log r\right)^2\\
				&\lesssim 1-\frac{1}{\log r}+ \left(\log\left(1+\frac{\log3}{\log r}\right)\right)^2.
		\end{aligned}\end{equation}
		By\eqref{eq:Ia} and \eqref{eq:Ib}, we have
		\begin{equation}\begin{aligned} \label{eq:IaIbIc}
				\lim\limits_{r\rightarrow+\infty} I_a + I_b +I_c \gtrsim -1.
		\end{aligned}\end{equation}
		Letting $r\rightarrow+\infty$ in \eqref{eq:main int} and combining \eqref{eq:leftside}, \eqref{eq:IaIbIc}, we derive a contradiction from the divergence of the left-hand side and the convergence of the right-hand side of \eqref{eq:main int}. Hence
		$$
		\int_{\mathbb{R}^3}|\nabla u|^2 + |\nabla B|^2 ~\mathrm{d}x =0,
		$$
		which implies $u\equiv u_0$ and $B\equiv B_0$ for all $x\in\mathbb{R}^3$, where ${u}_0$ and $B_0$  are two constants. By Lemma \ref{lem:est of 2-norm} with $\rho=\frac{R}{2}$, we have
		$$
		\|u\|_{L^2(B_{\frac{R}{2}})} + \|B\|_{L^2(B_{\frac{R}{2}})}   \lesssim R^{\frac{5 s-6}{6 s}}(\log R)^{\frac{s+6}{6s}},
		$$
		 which implies
		\begin{equation}\begin{aligned}
				|u_0| + |B_0| \lesssim R^{-\frac{2s+3}{3 s}}(\log R)^{\frac{s+6}{6s}},
		\end{aligned}\end{equation}
		for all $R>2$. Therefore, $u=B\equiv0$. The proof is completed.
	\end{proof}
	\begin{Rem}
It is evident that $0 < \frac{g(R)}{g(3R)} \leq 1$, but this is insufficient for estimating its lower bound. However, the result of Lemma \ref{lem:grow est} indicates that the growth rate of $g(R)$ is well-controlled, suggesting that its value might be increased through integration using the Froullani integral.
	\end{Rem}

	\noindent {\bf Acknowledgments.}
	The authors would like to thank Professors Wendong Wang for some helpful communications and the support from the Key R\&D Program Project No. 2023YFA1009200.
	\\
	
	\noindent {\bf Declaration of competing interest.}
	The authors state that there is no conflict of interest.\\
	
	\noindent {\bf Data availability.}
	No data was used in this paper.

	\bibliographystyle{amsplain}
	\bibliography{SV_MHD.bib}
\end{document}